\documentclass{amsart}

\usepackage{amssymb}

\usepackage{xy}
\xyoption{all}

\newtheorem{lem}{Lemma}[section]
\newtheorem{theorem}[lem]{Theorem}
\newtheorem{proposition}[lem]{Proposition}

\newtheorem{corollary}[lem]{Corollary}
\newtheorem{problem}[lem]{Problem}

\theoremstyle{definition}

\newcommand{\F}{\mathcal {F}}
\newcommand{\Ra}{\Rightarrow}
\newcommand{\U}{\mathcal U}
\newcommand{\V}{\mathcal V}

\newcommand{\A}{\mathcal A}

\newcommand{\LL}{\mathcal L}

\newcommand{\Fil}{\varphi}
\newcommand{\w}{\omega}

\newcommand{\upspace}{\upsilon}
\newcommand{\la}{\langle}
\newcommand{\ra}{\rangle}
\newcommand{\uupspace}{\upsilon^\bullet}

\title{Algebra in the superextensions of semilattices}
\author{Taras Banakh and Volodymyr Gavrylkiv}
\address{Ivan Franko National University of Lviv, Ukraine and\newline
Uniwersytet Humanistyczno-Przyrodniczy Jana Kochanowskiego, Kielce, Poland}
\email{t.o.banakh@gmail.com}
\address{Vasyl Stefanyk Precarpathian National University,
Ivano-Frankivsk, Ukraine}
\email{vgavrylkiv@yahoo.com}

\begin{document}
\begin{abstract} Given a semilattice $X$ we study the algebraic properties of the semigroup $\upspace(X)$ of upfamilies on $X$. The semigroup $\upspace(X)$ contains the Stone-\v Cech extension $\beta(X)$, the superextension $\lambda(X)$, and the space of filters $\varphi(X)$ on $X$ as closed subsemigroups. We prove that $\upspace(X)$ is a semilattice iff $\lambda(X)$ is a semilattice iff $\varphi(X)$ is a semilattice iff the semilattice $X$ is  finite and linearly ordered. We prove that the semigroup $\beta(X)$ is a band if and only if $X$ has no infinite antichains, and the semigroup $\lambda(X)$ is commutative if and only if $X$ is a bush with finite branches.
\end{abstract}
\maketitle

\section*{Introduction}

One of powerful tools in the modern Combinatorics of Numbers is
the method of ultrafilters based on the fact that each
(associative) binary operation $*:X\times X\to X$ defined on a
discrete topological space $X$ extends to a
right-topological (associative) operation $*:\beta (X)\times \beta
(X)\to\beta (X)$ on the Stone-\v Cech compactification $\beta (X)$
of $X$, see \cite{HS}, \cite{P}.
 The Stone-\v Cech extension $\beta (X)$ is the space
of ultrafilters on $X$. The extension of the operation from $X$ to
$\beta(X)$ can be defined by the simple formula:
\begin{equation}\label{extension}\U\ast\V=\big\la\bigcup_{x\in U}x{*}V_x:U\in\U,\;\;(V_x)_{x\in U}\in \V^U\big\ra,
\end{equation}
where $\la \mathcal B\ra=\{A\subset X:\exists B\in\mathcal B\;\;B\subset A\}$ is the upper closure of a family $\mathcal B$. In this case $\mathcal B$ is called a {\em base} of $\la\mathcal B\ra$.

Endowed with the so-extended operation, the Stone-\v Cech
compactification $\beta(X)$ becomes a compact right-topological
semigroup. The algebraic properties of this semigroup (for
example, the existence of idempotents or minimal left ideals) have
important consequences in combinatorics of numbers, see \cite{HS},
\cite{P}.

In \cite{G2} it was observed that the binary operation $*$ extends
not only to $\beta (X)$ but also  to the space $\upsilon(X)$ of all
upfamilies on $X$. By definition, a family $\F$ of
non-empty subsets of a discrete space $X$ is called an {\em
upfamily} if for any sets $A\subset B\subset X$ the inclusion $A\in\F$ implies $B\in\F$. The space $\upspace(X)$ is a closed subspace of the double power-set $\mathcal P(\mathcal P(X))$ endowed with the compact Hausdorff topology of the Tychonoff power $\{0,1\}^{\mathcal P(X)}$. In the papers \cite{G1}, \cite{G2}, \cite{BGN}--\cite{BG4} the space $\upsilon(X)$ was denoted by $G(X)$ and its elements were called inclusion hyperspaces\footnote{We decided to change the terminology and notation after discovering the paper \cite[2.7.4]{SS} that discusses monadic properties of the up-set functor $\upspace$.}. The extension of a binary
operation $\ast$ from $X$ to $\upsilon(X)$ can be defined in the same way
as for ultrafilters, i.e., by the formula~(\ref{extension})
applied to any two upfamilies $\U,\V\in \upsilon(X)$. If $X$
is a semigroup, then $\upsilon(X)$ is a compact Hausdorff
right-topological semigroup containing $\beta (X)$ as closed
subsemigroups. The algebraic properties of this semigroups were
studied in details in \cite{G2}.


The space $\upspace(X)$ of upfamilies over a discrete space $X$ contains many
interesting subspaces. First we recall some definitions. An upfamily $\A\in \upspace(X)$ is defined to be
\begin{itemize}
\item {\em a filter} if
$A_1\cap A_2\in\A$ for all sets $A_1,A_2\in\A$;
\item {\em an ultrafilter} if $\A=\A'$ for any filter $\A'\in \upspace(X)$ containing $\A$;
\item {\em linked} if $A\cap B\ne\emptyset$ for any sets $A,B\in\A$;
\item {\em maximal linked} if $\A=\A'$ for any linked
upfamily $\A'\in \upspace(X)$ containing $\A$.
\end{itemize}

By $\Fil(X)$, $\beta(X)$, $N_2(X)$, and $\lambda(X)$  we denote the subspaces of $\upspace(X)$ consisting of filter, ultrafilters, linked upfamilies, and maximal linked upfamilies, respectively. The space $\lambda(X)$ is
called {\em the superextension} of $X$, see \cite{vM}, \cite{Ve}. In \cite{G2} it was
observed that for a discrete semigroup $X$ the subspaces $\Fil(X)$, $\beta(X)$, $N_2(X)$, $\lambda(X)$ are closed subsemigroups of
the semigroup $\upspace(X)$.
The following diagram describes the inclusion relations between
these subspaces of $\upspace(X)$ (an arrow $A\to B$ indicates that $A$ is a subset of $B$).
$$
\xymatrix{
\beta(X)\ar[d]\ar[r]&\lambda(X)\ar[d]\\
\varphi(X)\ar[r]&N_{2}(X)\ar[r]&\upspace(X)
}$$


In \cite{G2}, \cite{BGN} --- \cite{BG4} we studied the properties of the compact right-topological semigroup $\upsilon(X)$ and its subsemigroups for groups $X$. In this paper we shall study the algebraic structure of the semigroups $\lambda(X)$, $\varphi(X)$, $N_2(X)$, and $\upsilon(X)$ for semilattices $X$.

Let us recall that a {\em semilattice} is a commutative idempotent semigroup.
Idempotent semigroups are called {\em bands}. So, in a band each element $x$ is an {\em idempotent}, which means that $xx=x$.
A semigroup $S$ is {\em linear} if $xy\in\{x,y\}$ for any elements $x,y\in X$. It follows that each linear semigroup $S$ is a band. Each (linear) semilattice is  partially (linearly) ordered by the relation $\le$ defined by $x\le y$ iff $xy=x$.

A semigroup $S$ is {\em cancellative} if for each element $a\in S$ the left shift $l_a:S\to S$, $l_a:x\mapsto ax$, and the right shift $r_a:S\to S$, $r_a:x\mapsto xa$, are injective.
A semigroup $S$ is called {\em Clifford} (resp. {\em sub-Clifford}) if $S$ is a union of groups (resp. of cancellative semigroups). Observe that a subsemigroup of a sub-Clifford semigroup is sub-Clifford and a finite semigroup $S$ is Clifford if and only if it is sub-Clifford. It is easy to see that a semigroup $S$ is sub-Clifford if and only if for every natural numbers $n,m$ it is {\em $(n,m)$-Clifford} in the sense that for any element $x\in S$ the equality $x^{n+1}=x^{m+1}$ implies $x^n=x^m$.

A semigroup $S$ is called {\em a regular semigroup} if $a\in aSa$ for any $a\in S$. Such a semigroup
$S$ is called {\em an inverse semigroup} if $ab=ba$ for any
idempotents $a,b\in S$. Observe that each band is a Clifford
semigroup and every Clifford semigroup is sub-Clifford and regular. An inverse semigroup with a unique idempotent is a group.

These algebraic properties relate as follows:
$$
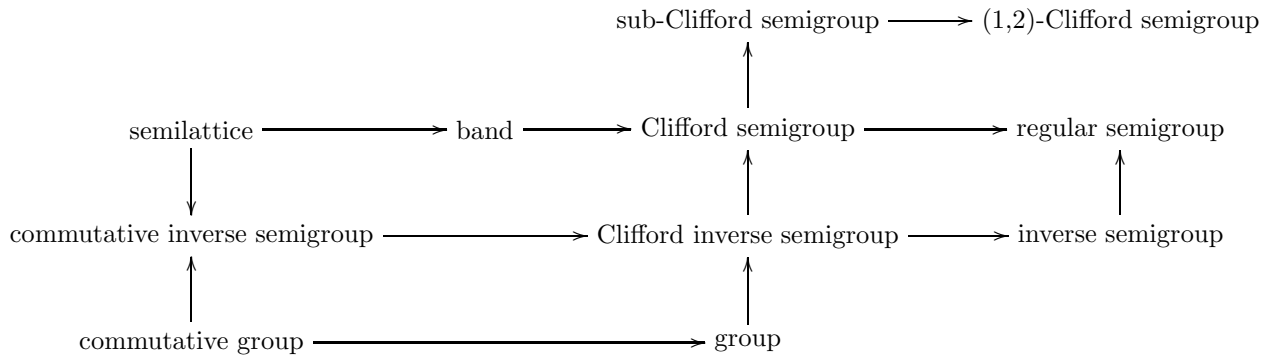
\xymatrix{
&&\mbox{sub-Clifford semigroup}\ar[r]&\mbox{(1,2)-Clifford semigroup}\\
\mbox{semilattice}\ar[r]\ar[d]&\mbox{band}\ar[r]&\mbox{Clifford semigroup}\ar[u]\ar[r]&\mbox{regular semigroup}\\
\mbox{commutative inverse semigroup}\ar[rr]&&\mbox{Clifford inverse semigroup}\ar[r]\ar[u]&\mbox{inverse semigroup}\ar[u]\\
\mbox{commutative group}\ar[rr]\ar[u]&&\mbox{group}\ar[u]
}
$$

In this paper we shall characterize semigroups $X$ whose extensions $\upspace(X)$, $\lambda(X)$, $\varphi(X)$ or $N_2(X)$ are bands, linear semigroups, commutative semigroups, or semilattices. In Section~\ref{s:lattice} we shall characterize lattices $X$ whose extensions $\upspace(X)$, $\lambda(X)$, $\varphi(X)$ are lattices. The results obtained in this paper will be applied in the paper \cite{BG5} devoted to the superextensions of inverse semigroups. 

\section{Semigroups whose extensions are bands}

In this section we shall characterize semigroups $X$ whose extensions $\upspace(X)$, $\lambda(X)$ or $\varphi(X)$ are bands.
Let us recall that a semigroup $S$ is a (linear) band if $xx=x$ for all $x\in X$ (and $xy\in\{x,y\}$ for all $x,y\in X$).

Let us recall that an element $a$ of a semigroup $S$ is {\em regular} in $S$ if $a\in aSa$. It is clear that each idempotent is a regular element.

\begin{theorem}\label{t1.1} For a semigroup $X$ the following conditions are equivalent:
\begin{enumerate}
\item $X$ is linear;
\item $\upspace(X)$ is a band;
\item $\varphi(X)$ is a band;
\item $\lambda(X)$ is a band.
\end{enumerate}
\end{theorem}

\begin{proof} $(1)\Ra(2)$ Assume that the semigroup $X$ is linear. To show that $\upspace(X)$ is a band, we should check that $\A*\A=\A$ for any upfamily  $\A\in\upsilon(X)$. Since $X$ is linear, for any $A\in\A$ we get $A=A*A\in\A*\A$ and hence $\A\subset\A*\A$.

To show that $\A\supset\A*\A$, fix any basic subset $B=\bigcup\limits_{x\in A}x{*}A_x\in\A*\A$ where $A\in\A$ and $A_x\in\A$ for all $x\in A$.

Now we consider two cases.

(i) There is $x\in A$ such that $xa=a$ for all $a\in A_x$. In this case $\A\ni A_x=x{*}A_x\subset B$ and thus $B\in\A$.

(ii) For every $x\in A$ there is $a\in A_x$ such that $xa\ne a$ and hence $xa=x$ (as $X$ is linear). In this case $\A\ni A\subset\bigcup_{x\in A}x*A_x=B$ and hence $B\in\A$.
\smallskip

The implications $(2)\Ra(3,4)$ are trivial.
\smallskip

$(3)\Ra(1)$ Assume that $\varphi(X)$ is a band. Then $X$, being a subsemigroup of $\varphi(X)$, also is a band. To show that $X$ is linear, take any two points $x,y\in X$ and consider the filter $\F=\la\{x,y\}\ra\in\varphi(X)$. Being an idempotent, the filter $\F$ is regular in $\upspace(X)$. Consequently, we can find an upfamily $\A\in\upsilon(X)$ such that $\F*\A*\F=\F$. It follows that there are sets $A_x,A_y\in\A$ such that $(xA_x\cup yA_y)\cdot\{x,y\}\subset\{x,y\}$. In particular, for every $a_x\in A_x$ we get $xa_xy\in\{x,y\}$. If $xa_xy=x$, then $xy=xa_xyy=xa_xy=x$. If $xa_xy=y$, then $xy=xxa_xy=xa_xy=y$, witnessing that the band $X$ is linear.

$(4)\Ra(1)$ Assume that $\lambda(X)$ is a band. Then $X$, being a subsemigroup of $\lambda(X)$, is a band as well. Assuming that the band $X$ is not linear, we can find two points $x,y\in X$ such that $xy\notin\{x,y\}$.
We claim that the maximal linked system $\LL=\la\{x,y\},\{x,xy\},\{y,xy\}\ra\in\lambda(X)$ is not an idempotent. We shall prove more: the element $\LL$ is not regular in the semigroup $\upspace(X)$.
Assuming the converse, we can find an upfamily $\A\in\upspace(X)$ such that $\LL*\A*\LL=\LL$. It follows from $\{x,y\}\in\LL=\LL*\A*\LL$ that $\{x,y\}\supset \bigcup_{u\in L}u*B_u$ for some set $L\in\LL$ and some sets $B_u\in\A*\LL$, $u\in L$. The linked property of family $\LL$ implies that the intersection $L\cap\{x,xy\}$ contains some point $u$. Now for the set $B_u\in\A*\LL$ find a set $A\in\A$ and a family $(L_a)_{a\in A}\in\LL^A$ such that $B_u\supset\bigcup_{a\in A}a*L_a$. Fix any point $a\in A$ and a point $v\in L_a\cap\{y,xy\}$. Then $uav\in uaL_a\subset uB_u\subset \{x,y\}$.
Since $u\in\{x,xy\}$ and $v\in \{y,xy\}$, the element $uav$ is equal to $xby$ for some element $b\in\{a,ya,ax,yax\}$. So, $xby\in\{x,y\}$. If $xby=x$, then $xy=xbyy=xby=x\in\{x,y\}$. If $xby=y$, then $xy=xxby=xby=y\in\{x,y\}$. In both cases we obtain a contradiction with the choice of the points $x,y\notin\{x,y\}$.
\end{proof}

 Observe that the proof of Theorem~\ref{t1.1} yields a bit more, namely:

\begin{proposition} For a band $X$ the following conditions are equivalent:
\begin{enumerate}
\item $X$ is linear;
\item each element of $\varphi(X)$ is regular in $\upspace(X)$;
\item each element of $\lambda(X)$ is regular in $\upspace(X)$.
\end{enumerate}
\end{proposition}

The linearity of a semilattice $X$ can be also characterized via the $(1,2)$-Clifford property of the semigroups $\varphi(X)$ and $\lambda(X)$.

\begin{theorem}\label{t(1,2)} For a semilattice $X$ the following conditions are equivalent:
\begin{enumerate}
\item $X$ is linear;
\item $\varphi(X)$ is $(1,2)$-Clifford;
\item $\lambda(X)$ is $(1,2)$-Clifford.
\end{enumerate}
\end{theorem}

\begin{proof} The implications $(1)\Ra(2,3)$ follow from Theorem~\ref{t1.1} because each band is a $(1,2)$-Clifford semigroup.
\smallskip

$(2,3)\Ra(1)$ Assume that the semilattice $X$ is not linear. Then $X$ contains two elements $x,y\in X$ such that $yx=xy\notin\{x,y\}$.

Consider the filter $\F=\la\{x,y\}\ra$ and observe that $\F\ne\F\cdot\F=\la\{x,xy,y\}\ra=\F\cdot\F\cdot\F$, which means that the semigroup $\varphi(X)$ is not $(1,2)$-Clifford.

To see that $\lambda(X)$ is not $(1,2)$-Clifford, consider the maximal linked system
$\LL=\la\{x,y\},\{x,xy\},\{y,xy\}\ra\in\lambda(X)$ and observe that $\LL\ne\LL\cdot\LL=\la\{xy\}\ra=\LL\cdot\LL\cdot\LL$.
\end{proof}

Next we characterize semigroups $X$ whose Stone-\v Cech extension $\beta(X)$ is a band. A sequence $(x_n)_{n\in\w}$ of points of some set $X$ is called {\em injective} if $x_n\ne x_m$ for any distinct numbers $n,m\in\w$.

\begin{theorem}
For a band  $X$ the semigroup $\beta(X)$ is a band if and
only if for each injective sequence $(x_n)_{n\in\w}$ in $X$ there are numbers $n<m$ such that $x_nx_m\in\{x_n,x_m\}$.
\end{theorem}

\begin{proof} To prove the ``only if'' part, assume that $(x_n)_{n\in\w}$ is an injective sequence in $X$ such that $x_nx_m\notin\{x_n,x_m\}$ for all $n<m$. We claim that there is an infinite subset $\Omega\subset \w$ such that $x_nx_m\ne x_k$ for any numbers $n,m,k\in\Omega$ with $n<m$.
For this we shall apply the famous Ramsey Theorem. Consider the 4-coloring $\chi:[\w]^3\to 4=\{0,1,2,3\}$ of the set $[\w]^3=\{(k,n,m)\in\w^3:k<n<m\}$, defined by
$$\chi(k,n,m)=\begin{cases}
1 &\mbox{if $x_kx_n=x_m$},\\
2 &\mbox{if $x_kx_m=x_n$},\\
3 &\mbox{if $x_nx_m=x_k$},\\
0 &\mbox{otherwise}.
\end{cases}
$$By the Ramsey Theorem \cite[5.1]{P}, there is an infinite set $\Omega\subset\w$ such that $\chi(\Omega^3\cap[\w]^3)$ is a singleton. It follows from the definition of the coloring $\Omega$ that this singleton is $\{0\}$, which means that for any numbers $k,n,m\in\Omega$ with $n<m$ and $k\notin\{n,m\}$ we get $x_nx_m\ne x_k$. Since $x_nx_m\notin\{x_n,x_m\}$ for any numbers $n<m$, we conclude that $x_nx_m\ne x_k$ for any numbers $k,n,m\in\Omega$ with $n<m$.

Now take any free ultrafilter $\A$ that contains the set $A=\{x_n\}_{n\in\Omega}$. Then for every $n\in\w$ the set $A_{>n}=\{x_m:n<m\in\Omega\}$ belongs to the ultrafilter $\A$. The choice of the sequence $A=\{x_n\}_{n\in\Omega}$ guarantees that $A\cap \bigcup_{n\in\Omega}x_n*A_{>n}=\emptyset$, which implies that $\A\ne\A*\A$ and hence the ultrafilter $\A$ is not an idempotent in $\beta(X)$.
\smallskip

To prove the ``if'' part, assume that $\beta(X)$ is not a band and find an ultrafilter $\F\in\beta(X)$ with $\F*\F\neq\F$. In particular,
$\F*\F\nsubseteq\F$. This implies that for some $A\in\F$ and
$\{A_x\}_{x\in A}\subset\F$ the set $\bigcup_{x\in
A}x{*}A_x\notin\F$.

Consider the set $X^{\uparrow}_{\F}=\{x\in X: {\uparrow}x\in\F\}$ where ${\uparrow}x=\{y\in X:xy=x\}$.
We claim that $X^{\uparrow}_{\F}\notin\F$. Assuming that
$X^{\uparrow}_{\F}\in\F$, we conclude that $A\cap X^{\uparrow}_{\F}\in\F$.
This implies that ${\uparrow}a\in\F$ and ${\uparrow}a\cap
A_a\in\F$ for any $a\in A\cap X^{\uparrow}_{\F}$. Therefore
$a*({\uparrow}a\cap A_a)=\{a\}$ and hence $$\bigcup_{x\in
A}x*A_x\supset\bigcup_{x\in A\cap
X^{\uparrow}_{\F}}x*({\uparrow}x\cap A_x)=\bigcup_{x\in A\cap
X^{\uparrow}_{\F}}\{x\}=A\cap X^{\uparrow}_{\F}\in\F.$$ Thus
$\bigcup_{x\in A}x{*}A_x\in\F$. This contradiction shows that $X^{\uparrow}_{\F}\notin\F$.

Next, consider the set $X^{\downarrow}_{\F}=\{x\in X:
{\downarrow}x\in\F\}$ where ${\downarrow}x=\{y\in X:xy=y\}$. We claim that
$X^{\downarrow}_{\F}\notin\F$. Assume that
$X^{\downarrow}_{\F}\in\F$. Then $A\cap X^{\downarrow}_{\F}\in\F$.
This implies that ${\downarrow}a\in\F$ and ${\downarrow}a\cap
A_a\in\F$ for any $a\in A\cap X^{\downarrow}_{\F}$. Therefore
$${\downarrow}a\cap A_a\subset a*({\downarrow}a\cap A_a)\subset a*A_a\subset   \bigcup_{x\in A}x*A_x.$$ Thus
$\bigcup_{x\in A}x{*}A_x\in\F$. This contradiction shows that $X^{\downarrow}_{\F}\notin\F$.

Since $\F$ is an ultrafilter, $X^{\uparrow}_{\F}\cup X^{\downarrow}_{\F}\notin\F$ and
$Z_{\F}=X\setminus(X^{\uparrow}_{\F}\cup
X^{\downarrow}_{\F})\in\F$. Let $x_0\in Z_{\F}$ be arbitrary and by induction, for every $n\in\w$ choose a point
$x_{n+1}\in Z_{\F}\setminus\bigcup_{i\leq
n}({\uparrow}x_i\cup{\downarrow}x_i)\in\F$.
 Then the injective sequence $(x_n)_{n\in\w}$ has the required property:
$x_nx_m\notin\{x_n,x_m\}$ for $n<m$ (which follows from $x_m\notin {\downarrow}x_n\cup{\uparrow}x_n$).
\end{proof}

A subset $A$ of a semigroup $X$ is called an {\em antichain} if $ab\notin\{a,b\}$ for any distinct points $a,b\in A$.
Theorem implies the following characterization:

\begin{corollary} For a semilattice  $X$ the semigroup $\beta(X)$ is a band if and only if each antichain in $X$ is finite.
\end{corollary}

\section{Semilattices whose extensions are commutative}

In this section we recognize the structure of semilattices $X$ whose extensions $\upspace(X)$, $N_2(X)$ or $\lambda(X)$ are commutative.

Commutative semigroups of ultrafilters were characterized in \cite[4.27]{HS} as follows:

\begin{theorem}\label{t2.4} The Stone-\v Cech extension $\beta(X)$ of a semigroup $S$ is not commutative if and only if there are sequences $(x_n)_{n\in\w}$ and $(y_n)_{n\in\w}$ in $X$ such that $\{x_ky_n:k<n\}\cap\{y_kx_n:k<n\}=\emptyset$.
\end{theorem}

This characterization implies the following (well-known) fact:

\begin{corollary}\label{c2.5} If the Stone-\v Cech extension $\beta(X)$ of a semilattice $X$ is commutative, then each linear subsemigroup in $X$ in finite.
\end{corollary}

\begin{proof} Assume conversely that $X$ contains an infinite linear subsemilattice $L$. Being linear, $L$ is linearly ordered by the order $\le$ defined by $x\le y$ iff $xy=x$. Since $L$ is infinite, we can apply Ramsey Theorem in order to find an injective sequence $(z_n)_{n\in\w}$ in $L$,  which is either strictly increasing or strictly decreasing. Put $x_n=z_{2n}$ and $y_n=z_{2n+1}$ for $n\in\w$. Applying Theorem~\ref{t2.4} to the sequences $(x_n)_{n\in\w}$ and $(y_n)_{n\in\w}$ we conclude that the semigroup $\beta(L)$ is not commutative. Then $\beta(X)$ is not commutative neither.
\end{proof}

In spite of Theorem~\ref{t2.4} the following problem seems to be open.

\begin{problem} Describe the structure of a semilattice $X$ whose Stone-\v Cech extension $\beta(X)$ is commutative.
\end{problem}

A similar problem on commutativity of semigroups $\upspace(X)$ also is open:

\begin{problem} Characterize semigroups $X$ whose extension $\upspace(X)$ is commutative.\newline {\rm (It can be shown that if $\upspace(X)$ is commutative, then $X$ is a commutative semigroup with finite linear idempotent band $E=\{x\in X:xx=x\}$ and $x^3=x^4$ for all $x\in X$).}
\end{problem}

We shall resolve this problem for bands. First we prove a useful result on multiplication of upfamilies on linear semigroups.

For a semigroup $X$ denote by $\uupspace(X)$ the subsemigroup of $\upspace(X)$ consisting of all upfamilies $\A\in\upspace(X)$ such that for each set $A\in\A$ there is a finite subset $F\in\A$ with $F\subset A$.

For a semigroup $X$ and two upfamilies $\A,\mathcal B\in\upspace(X)$ let $$\A\otimes\mathcal B=\la A*B:A\in\A,\;B\in\mathcal B\ra.$$
It is clear that $\A\otimes\mathcal B\subset\A*\mathcal B$. In the following theorem we show that for finite linear semigroups the converse inclusion also holds.

\begin{theorem}\label{t2.1} If $X$ is a linear semigroup, then $\A*\mathcal B=\A\otimes\mathcal B$ for any upfamilies $\A\in\uupspace(X)$ and $\mathcal B\in\upsilon(X)$.
\end{theorem}

\begin{proof} On the semigroup $X$ consider the relation $\le$ defined by:  $x\le y$ iff $yx=x$. This relation is reflexive and transitive. For a subsets $A\subset X$ and a point $x\in X$ we write $A\le x$ if $a\le x$ for all $a\in A$. It follows from the definition of the semigroup operation $*$ on $\upspace(X)$ that $\A\otimes\mathcal B\subset\A*\mathcal B$. To prove the reverse inclusion, fix any basic set $C=\bigcup_{a\in A}a{*}B_a\in\A*\mathcal B$ where $A\in\A$ and $B_a\in\mathcal B$ for all $a\in A$. Since $\A\in\uupspace(X)$, we can assume that the set $A$ is finite and hence can be enumerated as $A=\{a_1,\dots,a_n\}$ so that $a_i\le a_{i+1}$ for all $i<n$.
Now let us consider two cases.

1. For some $i\le n$ we get $B_{a_i}\le a_i$, which means that $a_ib=b$ for all $b\in B_{a_i}$ and hence $a_i*B_{a_i}=B_{a_i}$. For every $j\ge i$ the inequality $B_{a_i}\le a_i\le a_j$ implies $a_j*B_{a_i}=B_{a_i}$. Consequently, $A*B_{a_i}\subset\{a_1,\dots,a_{i-1}\}\cup B_{a_i}$.

We can assume that $i$ is the smallest number such that $B_{a_i}\le a_i$. 
In this case the minimality of $i$ implies that $B_{a_j}\not\le a_j$ for all $j<i$. This means $b_j\not\le a_j$ for some $b_j\in B_{a_j}$ and hence $a_jb_j=a_j$ (as $a_jb_j\in\{a_j,b_j\}$ and $a_jb_j\ne b_j$). Then $a_j{*}B_{a_j}\ni a_jb_j=a_j$ and thus $A*B_{a_i}\subset\{a_1,\dots,a_{i-1}\}\cup B_{a_i}\subset\bigcup_{j=1}^na_jB_{a_j}$, which implies that $C\in\A\otimes\mathcal B$.

2. $B_{a_i}\not\le a_i$ for all $i\le n$. In this case $a_i\in a_i*B_{a_i}$ for all $i$.
Observe that for any $b\in B_{a_n}$ and $i\le n$ we get $a_ib\in\{a_i,b\}$ by the linearity of $X$.
If $a_ib\ne a_i$, then $a_ib=b$ and $a_ib=b=a_na_ib=a_nb=\in a_nB_{a_n}$. So,
$$\A\otimes \mathcal B\ni A*B_{a_n}\subset \{a_1,\dots,a_n\}\cup a_nB_{a_n}\subset \bigcup_{i=1}^na_iB_{a_i}=C$$ and hence  $C\in\A\otimes\mathcal B$.
\end{proof}

Now we are able to characterize bands $X$ with commutative extensions $\upspace(X)$ and $N_2(X)$.

\begin{theorem}\label{t2.2} For a band $X$ the following conditions are equivalent:
\begin{enumerate}
\item $X$ is a finite linear semilattice;
\item the semigroup $\upsilon(X)$ is commutative;
\item the semigroup $N_2(X)$ is commutative;
\item the semigroup $\lambda(X)$ is commutative and $(1,2)$-Clifford.
\end{enumerate}
\end{theorem}

\begin{proof} The implication $(1)\Ra(2)$ follows from Theorem~\ref{t2.1} as
$\A*\mathcal B=\A\otimes\mathcal B=\mathcal B\otimes \A=\mathcal B*\A$ for every $\A,\mathcal B\in\upspace^\bullet(X)=\upspace(X)$.
\smallskip

The implication $(2)\Ra(3)$ is trivial.
\smallskip

$(3)\Ra(1)$ Assume that the semigroup $N_2(X)$ is commutative. Then so is the semigroup $X$. Being a commutative band, the semigroup $X$ is a semilattice. Assuming that $X$ is not linear, we can find two points $x,y\in X$ with $xy\notin\{x,y\}$. It can be shown that
the linked upfamilies $\A=\langle \{x,y\}\rangle$ and
$\mathcal B=\langle \{x,xy\}, \{y,xy\}\rangle\in N_k(X)$ do not commute because $\{xy\}\in\A{*}\mathcal B\setminus\mathcal B{*}\A$.
Therefore, $X$ is a linear semilattice. Since $\beta(X)\subset\upspace(X)$ is commutative, Corollary~\ref{c2.5} implies that the linear semilattice $X$ is finite.
\smallskip

$(1)\Leftrightarrow(4)$ If $X$ is a finite linear semilattice, then $\lambda(X)$ is commutative by the  implication $(1)\Ra(2)$ of this theorem and is $(1,2)$-Clifford by Theorem~\ref{t(1,2)}.

If the semigroup $\lambda(X)$ is commutative and $(1,2)$-Clifford, then the semigroup  $X\subset\lambda(X)$ is commutative and by Theorem~\ref{t(1,2)}, $X$ is linear. By Corollary~\ref{c2.5}, the linear semilattice $X$ is finite.
\end{proof}

Now we shall characterize semilattices $X$ with commutative superextension $\lambda(X)$. A semilattice $X$ is called a {\em bush} if for any maximal linear subsemilattices $A,B\subset X$ the product $A*B$ is the singleton $\{\min X\}$ containing the smallest element $\min X$ of $X$. This definition implies that $A\cap B=A*B=\{\min X\}$. By a {\em branch} of a bush $X$ we understand a maximal linear subsemilattice of $X$.

\begin{theorem}\label{t2.6} A semilattice $X$ has commutative superextension $\lambda(X)$ if and only if $X$ is a bush with finite branches.
\end{theorem}

\begin{proof} First assume that $X$ is a bush with finite branches, and take any two maximal linked systems $\A,\mathcal B\in\lambda(X)$. Since the products $\A*\mathcal B$ and $\mathcal B*\A$ are maximal linked upfamilies, the equality $\A*\mathcal B=\mathcal B*\A$ will follow as soon as we check that any two basic sets $C_{AB}=\bigcup_{a\in A}a{*}B_a\in\A*\mathcal B$ and $C_{BA}=\bigcup_{b\in B}b{*}A_b\in\mathcal B*\A$ have non-empty intersection. Here $A\in\A$, $(B_a)_{a\in A}\in \mathcal B^A$, $B\in\mathcal B$, and $(A_b)_{b\in B}\in\A^B$.  Assume conversely that $C_{AB}\cap C_{BA}=\emptyset$. Then either $\min X\notin C_{AB}$ or $\min X\notin C_{BA}$.

Without loss of generality, $\min X\notin C_{AB}$. Then $\min X\notin A$ and for each $a\in A$ the set $\{a\}\cup B_a$ lies in a branch of $X$. Since branches of $X$ meet only at the point $\min X$, all the sets $\{a\}\cup B_a$, $a\in A$, lie in the same (finite) branch. Repeating the argument of Theorem~\ref{t2.1}, we can show that $C_{AB}\supset AB'$ for some set $B'\in\mathcal B$. Since $\mathcal B$ is linked, there is a point $b\in B\cap B'$. By the same reason, there is a point $a\in A\cap A_b$. Then $ab=ba\in AB'\cap bA_b\subset C_{AB}\cap C_{BA}$ and we are done.
\smallskip

Now assume that $X$ is a semilattice with commutative superextension  $\lambda(X)$. Corollary~\ref{c2.5} implies that all branches of $X$ are finite. We claim that for every $z\in X$ the lower set ${\downarrow}z=\{x\in X:xz=x\}$ is linear. Assuming the converse, find two points $x,y\in{\downarrow}z$ such that $xy\notin\{x,y\}$. It follows that the points $x,y,z,xy$ are pairwise distinct. It is easy to check that the maximal linked upfamilies $\A=\la\{x,y\},\{x,z\},\{y,z\}\ra$ and $\mathcal B=\la\{x,y\},\{x,xy\},\{y,xy\}\ra$ do not commute because $\{x,y\}\in\mathcal B*\A\setminus\A*\mathcal B$. Thus ${\downarrow}z$ is linear for every $z\in X$, which means that $X$ is a tree.

Assuming that the tree $X$ is not a bush, we can find two points $x,y\in X$ such that $xy\notin\{x,y,z\}$ where $z=\min X$. Now consider the maximal linked systems $\A=\la \{x,y\},\{x,z\},\{y,z\}\ra$ and $\mathcal B=\la\{x,y\},\{x,xy\},\{y,xy\}\ra$ and observe that they do not commute as $\{xy\}\in \A*\mathcal B$ misses the set $\{x,y,z\}\in\mathcal B*\A$.
\end{proof}

\section{Semigroups whose extensions are semilattices}

In this section we shall characterize semigroups $X$ whose extensions $\upspace(X)$, $\lambda(X)$, $\varphi(X)$, or $N_2(X)$ are semilattices.

\begin{theorem}\label{t3.1} For a semigroup $X$ the following conditions are equivalent:
\begin{enumerate}
\item $X$ is finite linear semilattice;
\item $\upsilon(X)$ is a semilattice;
\item $\lambda(X)$ is a semilattice;
\item $\varphi(X)$ is a semilattice.
\end{enumerate}
\end{theorem}

\begin{proof} $(1)\Ra(2)$ If $X$ is a finite linear semilattice, then $\upsilon(X)$ is a semilattice (=commutative band) by Theorems~\ref{t1.1} and \ref{t2.2}.
\smallskip

The implications $(2)\Ra(3,4)$ are trivial.
\smallskip

The implication $(3)\Ra(1)$ follows from Theorems~\ref{t1.1} and \ref{t2.6}.
\smallskip

$(4)\Ra(1)$ Assume that $\varphi(X)$ is a semilattice. Then $X$, being a subsemigroup of the commutative semigroup $\varphi(X)$ is commutative. Since $\varphi(X)$ is a band, $X$ is a linear semigroup by Theorem~\ref{t1.1}. Thus $X$, being a commutative linear semigroup, is a linear semilattice.
Since the subsemigroup $\beta(X)\subset\lambda(X)$ is commutative, the linear semilattice $X$ is finite by Corollary~\ref{c2.5}.
\end{proof}

\section{Semigroups whose extensions are linear}

In this section we characterize semigroups $X$ whose extensions $\upspace(X)$, $\lambda(X)$ or $\varphi(X)$ are linear semigroups.

A semigroup $S$ is called a {\em semigroup of left} ({\em right}) {\em zeros} if $xy=x$ (resp. $xy=y$) for all $x,y\in X$.

\begin{theorem}\label{t4.1} For a semigroup $X$ the semigroup $\upspace(X)$ is linear if and only if $X$ is either a semigroup of right zeros or a semigroup of left zeros.
\end{theorem}

\begin{proof} If $X$ is a semigroup of left zeros, then for any upfamilies $\A,\mathcal B\in\upspace(X)$ and any basic element $\bigcup_{x\in A}xB_x\in\A*\mathcal B$ we get $\bigcup_{x\in A}xB_x=\bigcup_{x\in A}\{x\}=A$ and thus $\A*\mathcal B\subset\A$. On the other hand, each $A\in\A$ belongs to $\mathcal A*\mathcal B$ as $A=A*B\in\A*\mathcal B$ for any $B\in\mathcal B$.

 Assume that the semigroup $\upspace(X)$ is linear. Then $X$, being a subsemigroup of $\upspace(X)$, also is linear. Let $x,y$ be any two distinct elements of $X$. First we prove that $xy\ne yx$.
Assume conversely that $xy=yx$. Then $xy=yx\in\{x,y\}$ and we lose no generality assuming that $xy=x$. Now consider two upfamilies $\A=\la\{x,y\}\ra$ and $\mathcal B=\la\{x\},\{y\}\ra$ and observe that
$$\mathcal B*\A=\la\{xx,xy\},\{yx,yy\}\ra=\la\{x\},\{x,y\}\ra=\la\{x\}\ra\notin\{\A,\mathcal B\},$$ so $\upsilon(X)$ is not linear and this is a required contradiction.

Thus $xy\ne yx$ for all distinct points $x,y\in X$. We call a pair $(x,y)\in X^2$ {\em left} if $xy=x$ and $yx=y$ and {\em right} if $xy=y$ and $yx=x$. Since $X$ is linear, each pair $(x,y)\in X^2$ is either left or right. We claim that either all pairs $(x,y)\in X^2$ are left or else all such pairs are right. Assuming the opposite, find pairs $(x,y),(a,b)\in X^2$ such that $(x,y)$ is not left and $(a,b)$ is not right. Then $x\ne y$, $a\ne b$ and the pair $(x,y)$ is right while $(a,b)$ is left. Consider the filters $\A=\la\{x,a\}\ra$ and $\mathcal B=\la\{y,b\}\ra$ and observe that $\A*\mathcal B=\la\{xy,xb,ay,ab\}\ra=\la\{y,xb,ay,a\}\ra$. Since $\upspace(X)$ is linear, either $\A*\mathcal B=\A$ or $\A*\mathcal B=\mathcal B$. In the first case $\{x,a\}\supset\{y,xb,ay,a\} \supset\{y,a\}$ and hence $y=a$. In the second case, $\{y,a\}\subset\{y,b\}$ and thus $a=y$. Now consider the filters $\mathcal C=\la \{x,b\}\ra$ and $\mathcal D=\la\{a\}\ra$ and observe that $\mathcal C*\mathcal D=\la\{xa,ba\}\ra=\la\{xy,b\}\ra=\la\{y,b\}\ra=\la\{a,b\}\ra\notin\{\mathcal C,\mathcal D\}$, which contradicts the linearity of $\upspace(X)$.

Therefore either each pair $(x,y)\in X^2$ is left and then $X$ is a semigroup of left zeros or else each pair $(x,y)\in X^2$ is right and then $X$ is a semigroup of right zeros.
\end{proof}

\begin{theorem}\label{t4.2} For a semigroup $X$ the following conditions are equivalent:
\begin{enumerate}
\item the semigroup $\varphi(X)$ is linear;
\item the semigroup $N_2(X)$ is linear;
\item either $X$ is a semigroup of left zeros or $X$ is a semigroup of right zeros or else $X$ is a semilattice of order $|X|\le 2$.
\end{enumerate}
\end{theorem}

\begin{proof} $(3)\Ra(2)$ If $|X|=1$, then $N_2(X)$ is a singleton and hence is a
linear semigroup. If $X$ is a semilattice of order $|X|=2$, then $X=\{0,1\}$ for some elements $0,1$ with $0\cdot 1=1\cdot 0=0$. In this case $N_2(X)=\varphi(X)$ is a 3-element linear semilattice ordered as:
$$\la \{0\}\ra\le \la\{0,1\}\ra\le\la\{1\}\ra.$$

If $X$ is a semigroup of left or right zeros, then the semigroup $\upsilon(X)$ is linear by Theorem~\ref{t4.1} and so is its subsemigroup $N_2(X)$.
\smallskip

$(2)\Ra(1)$ Is the semigroup $N_2(X)$ is linear, then so is its subsemigroup $\varphi(X)$.
\smallskip

$(1)\Ra(3)$ Assume that the semigroup $\varphi(X)$ is linear.
Then $X$, being a subsemigroup of $\varphi(X)$, is linear as well.
If $|X|\le 2$, then either $X$ is a linear semilattice or a semigroup of left or right zeros. So, we assume that $|X|\ge 3$. We claim that distinct elements $x,y\in X$ do not commute. Assume conversely that $xy=yx$ for some distinct elements $x,y\in X$. Since $xy=yx\in\{x,y\}$ we lose no generality assuming that $xy=yx=x$. Fix any element $z\in X\setminus\{x,y\}$. Now consider 3 cases:

1. $zx=z$. In this case we can consider the filters $\A=\la\{z,y\}\ra$ and $\mathcal B=\la\{x,y\}\ra$ and observe that $\A*\mathcal B=\la\{zx,yx,zy,yy\}\ra=\la \{z,x,zy,y\}\ra\notin\{\A,\mathcal B\}$, which contradicts the linearity of $\varphi(X)$.

2. $zx=x$ and $zy=z$.  In this case we can consider the filters $\A=\la\{z,y\}\ra$ and $\mathcal B=\la\{x,y\}\ra$ and observe that $\A*\mathcal B=\la\{zx,yx,zy,yy\}\ra=\la \{x,x,z,y\}\ra\notin\{\A,\mathcal B\}$, which contradicts the linearity of $\varphi(X)$.

3. $zx=x$ and $zy=y$.  In this case we can consider the filters $\A=\la\{x,z\}\ra$ and $\mathcal B=\la\{y,z\}\ra$ and observe that $\A*\mathcal B=\la\{xy,xz,zy,zz\}\}\ra=\la \{x,xz,y,z\}\ra\notin\{\A,\mathcal B\}$, which again contradicts the linearity of $\varphi(X)$.

Those contradictions show that distinct elements of $X$ do not commute. Continuing as in the proof of Theorem~\ref{t4.1}, we can show that $X$ is a semigroup of right or left zeros.
\end{proof}

Finally, we characterize commutative semigroups with linear superextensions.

\begin{theorem}\label{t4.3} For a commutative semigroup $X$ the semigroup $\lambda(X)$ is linear if and only if $X$ is a linear semilattice of order $|X|\le 3$.
\end{theorem}

\begin{proof} If $X$ is a linear semilattice of order $|X|\le 2$, then the semigroup $\lambda(X)=X$ is linear.

If $X$ is a linear semilattice of order $|X|=3$, then $X$ can be identified with the set $3=\{0,1,2\}$ endowed with the operation $xy=\min\{x,y\}$.
The semigroup $\lambda(X)$ contains 4 elements: $0,1,2$ and $\Delta=\{A\subset 3:|A|\ge 2\}$. One can check that $\lambda(3)$ is a linear semilattice ordered as follows:
$$0\le \Delta\le 1\le 2.$$

This proves the ``if'' part of the theorem. To prove the ``only if'' part we first shall analyze the structure of the superextension $\lambda(4)$ of the semilattice $4=\{0,1,2,3\}$ endowed with the operation $xy=\min\{x,y\}$.
By Theorem~\ref{t3.1}, $\lambda(4)$ is a semilattice. It contains 12 elements: $$\la k\ra,\;\;\Delta_k=\la \{A\subset n:|A|=2,\;k\notin A\}\mbox{ \  and \ $\square_k=\la\{n\setminus\{k\},A:A\subset n,\;|A|=2,\;k\in A\}\ra$ \ where $k\in 4$}.$$
The order structure of the semilattice $\lambda(4)$ is described in the following diagram:
$$
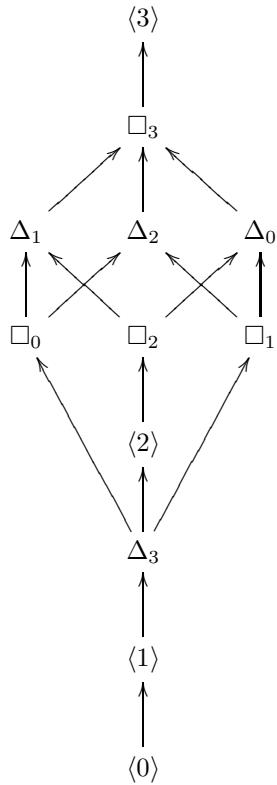
\xymatrix{
&\la 3\ra\\
&\square_3\ar[u]\\
\Delta_1\ar[ur]&\Delta_2\ar[u]&\Delta_0\ar[ul]\\
\square_0\ar[u]\ar[ur]&\square_2\ar[ul]\ar[ur]&\square_1\ar[u]\ar[ul]\\
&\la 2\ra\ar[u]\\
&\Delta_3\ar[u]\ar[uul]\ar[uur]\\
&\la1\ra\ar[u]\\
&\la0\ra\ar[u]
}
$$
Looking at this diagram we see that the semilattice $\lambda(4)$ is not linear.

Now assume that $X$ is a commutative semigroup whose superextension $\lambda(X)$ is linear. Then $X$ is a linear semilattice. If $|X|>3$, then $\lambda(X)$ is not linear as it contains a subsemigroup isomorphic to the semilattice $\lambda(4)$, which is not linear.
\end{proof}

\section{Lattices whose extensions are lattices}\label{s:lattice}

In this section we characterize lattices whose extensions $\upspace(X)$, $\lambda(X)$ or $\varphi(X)$ are lattices.

A {\em lattice} is a set $X$ endowed with two semilattice operations $\wedge,\vee:X\times X\to X$ such that $(x\wedge y)\vee y=y$ and $(x\vee y)\wedge y=y$ for all $x,y\in X$.

Both operations $\wedge$ and $\vee$ of a lattice $X$ can be extended to right-topological operations $\wedge$ and $\vee$ on the compact Hausdorff space $\upspace(X)$. Is it natural to ask if the triple $(\upspace(X),\wedge,\vee)$ is a lattice.

A lattice will be called {\em linear} if $x\wedge y,x\vee y\in\{x,y\}$ for all $x,y\in X$.

\begin{theorem} For a lattice $X$ the following conditions are equivalent:
\begin{enumerate}
\item $X$ is a linear lattice of order $|X|\le 2$.
\item $\upspace(X)$ is a lattice;
\item $\lambda(X)$ is a lattice;
\item $\varphi(X)$ is a lattice.
\end{enumerate}
\end{theorem}

\begin{proof} $(1)\Ra(2)$ If $X$ is a linear lattice of order $|X|=1$, then $\upspace(X)=X$ is a trivial lattice. If $X$ is a linear lattice of order 2, then $X$ can be identified with the lattice $2=\{0,1\}$ endowed with the operations $x\wedge y=\min\{x,y\}$ and $x\vee y=\max\{x,y\}$. In this case $\lambda(2)=\beta(2)$ coincides with the lattice $2$, $\varphi(2)=\{\la\{0\}\ra,\la\{0,1\}\ra,\la\{1\}\ra\}$ is a 3-element lattice, isomorphic to the lattice $3=\{0,1,2\}$ endowed with the operations $\min$ and $\max$, and $\upspace(2)=\big\{\la\{0\}\ra,\la \{0,1\}\ra,\la\{0\},\{1\}\ra,\la\{1\}\ra\big\}$ is a 4-element lattice isomorphic to the lattice $\{0,1\}^2$.
\smallskip

The implications $(2)\Ra(3,4)$ are trivial.
\smallskip

$(3,4)\Ra(1)$ Assume that $\lambda(X)$ or $\varphi(X)$ is a lattice. By Theorem~\ref{t3.1}, the lattice $X$ is finite and linear.  We claim that $|X|\le 2$. Assuming the converse, we conclude that the lattice $X$ contains a sublattice isomorphic to the lattice $(3,\min,\max)$.

Consider the maximal linked upfamily $\Delta=\{A\subset 3:|A|\ge 2\}$ and observe that $$\max\{\Delta,\la 1\ra\}=\la 1\ra=\min\{\Delta,\la 1\ra\},$$ which implies that $\lambda(3)$ is not a lattice and then $\lambda(X)$ also is not a lattice.

Next, consider the filters $\A=\la\{0,1,2\}\ra$ and $\mathcal B=\la\{0,2\}\ra$ and observe that $$\max\{\A,\mathcal B\}=\A=\min\{\A,\mathcal B\}$$implying that $\varphi(3)$ is not a lattice and then $\varphi(X)$ also cannot be a lattice.
\end{proof}

\end{document}